\documentclass[12pt]{amsart}
\usepackage{mathrsfs}
\usepackage{amssymb,amsmath}
\usepackage{setspace}
\usepackage{color}
\usepackage[all]{xy}
\usepackage{graphics}
\usepackage[english]{babel}
\usepackage[dvips]{sseq}
\calclayout
\newtheorem{dummy}{anything}[section]
\newtheorem{theorem}[dummy]{Theorem}

\newtheorem{lemma}[dummy]{Lemma}
\newtheorem{proposition}[dummy]{Proposition}
\newtheorem{corollary}[dummy]{Corollary}
\theoremstyle{definition}

 \newtheorem{remark}[dummy]{Remark}
 
 \newtheorem*{acknowledgement}{Acknowledgement}

\newcommand{\z}{\mathbb Z}
\newcommand{\w}{\widetilde}
\newcommand{\rp}{\mathbb R \mathrm P}
\newcommand{\cp}{\mathbb C \mathrm P}
\newcommand{\rpi}{\rp^{\infty}}
\newcommand{\cpi}{\cp^{\infty}}

\newcommand{\scs}{\, \sharp_ {S^1}}

\DeclareMathOperator{\Pin}{Pin}
\DeclareMathOperator{\Wh}{Wh}

\DeclareMathOperator{\Spin}{Spin}
\newcommand{\XX}[1]{X^5(#1)}
\newcommand{\HH}[1]{H\rp^5_{#1}}
\newcommand{\Sig}{\Sigma^5}

\newcommand{\OS}{\Omega^{\Spin}}

\newcommand{\vsp}{\phantom{$\Big ($}}

\begin{document}
\title{Free Involutions on $S^2 \times S^3$}
\author{Yang Su}
\address{Hua Loo-Keng Key Laboratory of Mathematics
\newline \indent
Chinese Academy of Sciences
\newline\indent
Beijing, 100190, China}

\address{Mathematisches Institut
\newline \indent 
Universit\"at M\"unster
\newline \indent
Einsteinstra\ss e 62, 48149, M\"unster, Germany}

\email{suyang{@}math.ac.cn}

\date{Dec.~16, 2010}

\begin{abstract}
In this paper, we classify smooth $5$-manifolds with fundamental
group isomorphic to $\z/2$ and universal cover diffeomorphic to $S^2
\times S^3$.This gives a classification of smooth free
involutions on $S^2 \times S^3$ up to conjugation.
\end{abstract}

\maketitle

\section{Introduction}\label{sec:one}
The study of symmetries of manifolds is an important topic of
topology and geometry. Free involutions (fixed point free $\z/2$-actions) may be thought as the simplest type of symmetries a manifold can possess. The study of free involutions on manifolds has a long history. For certain manifolds whose topology is simple, the classification of free involutions are known: free involutions on the spheres $S^n$ ($n \ge 5$) were classified in the PL- category by the classification of fake real projective spaces by Wall \cite{Wall}; free involutions on fake $\cp^n$ ($n \ge 3$)  were classified by Sady \cite{Sady}; topological free involutions on $S^1 \times S^n$ were recently classified by Jahren and Kwasik \cite{J-K}.

In dimension $5$, besides the cases $S^5$ and $S^1 \times S^4$ mentioned above,
Hambleton studied the existence of free involutions on
simply-connected spin $5$-manifolds \cite{Ham}. A classification of free involutions on simply-connected
$5$-manifolds with torsion free second homology and trivial action
on $H_2$ is obtain by Hambleton and Su in \cite{Ham-Su}. In this note we give a classification of smooth free involutions on $S^2 \times S^3$ up to conjugation. As the action on $H_2$ may be nontrivial, this work can be thought as a partial generalization of \cite{Ham-Su}.

\smallskip
The classification of free involutions on a manifold up to conjugation
is equivalent to the classification of the quotient spaces. In the
paper we study the classification of smooth $5$-manifolds with fundamental group
$\z/2$ and universal cover $S^2 \times S^3$, and henceforth obtain a
classification of smooth free involutions on $S^2 \times S^3$.

\smallskip

Before stating the main theorem, we give some examples of smooth
free involutions on $S^2 \times S^3$. First, there are some
obvious (linear) ones. Let $T_i$ be the free involution defined by
$$S^2 \times S^3 \to S^2 \times S^3, \ \ (x,y) \mapsto (-x, \tau_i (y))$$
where $-1$ is the
antipodal map on $S^2$ and $\tau_i$ is the reflection with $i$ $(-1)$-eigenvalues,
$i=0,1,2,3,4$. We denote the corresponding quotient manifold by
$Y_i$. Let $\eta_2$ be the canonical line bundle ove $\rp^2$, then clearly
$Y_i$ is the sphere bundle of the vector bundle $i\eta_2 \oplus
(4-i)\underline{\mathbb R}$. 

Similarly, let $T'_i$ be the free involution
$$S^2 \times S^3 \to S^2 \times S^3, \ \ (x,y) \mapsto (\tau'_i(x), -y)$$ 
where $-1$ is the antipodal map on $S^3$ and $\tau'_i$ is a reflection with $i$ $(-1)$-eigenvalues, $i=0,1,2,3$. We denote the corresponding quotient
manifold by $Z_j$. Let $\eta_3$ be
 the canonical line bundle ove $\rp^3$, then $Z_j$ is the sphere bundle of the vector bundle $j\eta_3 \oplus (3-j)\underline{\mathbb R}$

Another class of examples is given by the following construction.
$S^2 \times S^3$ can be realized as the link $\Sigma^5_q$ of a
Brieskorn singularity of type $A_q$, $q=0,2,4,6,8$. There is a
smooth free involution $T$ on $\Sigma^5_q$ induced by an involution
of the ambient space. We denote the quotient manifold by
$\Sigma^5_q/T$. 

A third class of examples is given by the $S^1$-connected sum
of $5$-dimensional fake projective spaces. By this construction we obtain
manifolds $X^5(q)$ ($q=0,2,4,6,8$) with $\pi_1(X^5(q)) \cong \z/2$
and universal cover $S^2 \times S^3$. Detailed description of these
two classes of manifolds will be given in \S 2C. These manifolds correspond to
nonlinear involutions on $S^2 \times S^3$.

\smallskip

Let $T \colon S^2 \times S^3 \to S^2 \times S^3$ be a smooth free
involution, $M^5 =S^2 \times S^3/T$ be the quotient space. Then
according to the action of $T$ on $\pi_2(S^2 \times S^3)$, as a
$\z[\z/2]$-module, $\pi_2(M)$ is isomorphic to $\z_+$ (the trivial
module) or $\z_-$ (the nontrivial module). The second space of the
Postnikov tower of $M$, $P_2(M)$, is a fibration over $\rp^{\infty}$
with fiber $\cp^{\infty}$. There are two such
fibrations in the $\z_-$-case, one admits a section and the other
one doesn't. We denote the former by $P$ and the latter by $Q$. More
precise information about $P$ and $Q$ will be given in \S 2A.

\begin{theorem}\label{thm:one}
Let $M^5$ be a smooth $5$-manifold with $\pi_1(M) \cong \z/2$ and universal
cover $\w{M} \cong S^2 \times S^3$.

\begin{enumerate}
\item If $\pi_2(M) \cong \z_+$, then $M$ is orientable, and the classification of $M$ up to diffeomorphism is given in the following table

\smallskip

\begin{center}
\begin{tabular}{|c|c|}
\hline
\vsp
          $w_2(M)=0  $          & $w_2(M)\ne 0$     \\[.3ex]
\hline
\vsp
          $S^2 \times \rp^3 $   & $X^5(q)$, \ \ $q=0,2,4,6,8$ \\[.3ex]
\hline
\end{tabular}
\end{center}
The manifolds $X^5(q)$ are classified by the $\Pin^+$-bordism class of their characteristic submanifold $N_q$ with $[N_q]=q \in \Omega_4^{\Pin^+}/\pm =\{0,1,\dots, 8\}$. 

\

\item If $\pi_2(M) \cong \z_-$ and $M$ is orientable, then $P_2(M)=Q$ and the classification of $M$ up to diffeomorphism is given in the following table
\smallskip

\begin{center}
\begin{tabular}{|c|c|}
\hline
\vsp
          $w_2(M)=0  $          & $w_2(M)\ne 0$     \\[.3ex]
\hline
\vsp
          $Y_1$, $Y_1'$                 & $\Sigma^5_q/T$, \ \ $q=0,2,4,6,8$ \\[.3ex]
\hline
\end{tabular}
\end{center}
The manifolds $\Sigma^5_q/T$ are classified by the $\Pin^+$-bordism class of their characteristic submanifold $N_q$ with $[N_q]=q \in \Omega_4^{\Pin^+}/\pm =\{0,1,\dots, 8\}$. The manifolds $Y_1$ and $Y_1'$ are distinguished by the following $KO$-theoretic invariant: let $ f \colon Y_1 \to Q$ and $f' \colon Y_1' \to Q$ be the second stage Postnikov map for $Y_1$ and $Y_1'$ respectively. There is a canonical element $u \in ko^4(Q)$ such that $\langle f^*(u), [Y_1]_{ko} \rangle =0 \in ko_1 = \z/2$ and $\langle f'^*(u), [Y_1']_{ko} \rangle =1 \in ko_1 = \z/2$, where $[Y_1]$ and $[Y_1']$ are the $KO$-fundamental class of $Y_1$ and $Y_1'$ respectively. Furthermore, $Y_{1}$ and $Y_{1}'$ are not homotopy equivalent.

\

\item If $\pi_2(M) \cong \z_-$ and $M$ is nonorientable, then the classification of $M$ up to diffeomorphism is given in the following table
\smallskip

\begin{center}
\begin{tabular}{|c|c|c|}
\hline
\vsp
                & $w_2(M)=0  $          & $w_2(M)\ne 0$     \\[.3ex]
\hline
\vsp
$P_2(M)=P$   & $Z_1 $   & $-$ \\[.3ex]
\hline
\vsp
$P_2(M)=Q$   & $Y_2$                 & $S^3 \times \rp^2$ \\[.3ex]
\hline
\end{tabular}
\end{center}
\end{enumerate}
\end{theorem}

Note that we have explained the construction of all the manifolds in the above tables except for $Y_1'$, which will be constructed in \S 4 using surgery.

\smallskip

As a corollary, we have a classification of smooth free involutions
on $S^2 \times S^3$.

\begin{corollary} \label{cor:one}
There are exactly $13$ orientation preserving smooth involutions on
$S^2 \times S^3$ and $3$ orientation reversing ones. The quotient
spaces are given in the above tables.
\end{corollary}

\begin{remark}It will be interesting to give explicit description of the involutions with quotient space $X^5(q)$ and $Y_1'$.

\end{remark}

\begin{remark}
There are topological free involutions on $S^2 \times S^3$ which are not conjugate to any smooth one. In
\cite{Ham-Su}, a $5$-manifold $*(S^2 \times \rp^3)$ is constructed.
This manifold is homotopy equivalent to $S^2 \times \rp^3$, but
doesn't admit any smooth structure. Therefore the
deck-transformation on the universal cover is a non-smoothable free
involution on $S^2 \times S^3$. The method of this paper is also
valid for the classification in the topological category.
\end{remark}

The remaining part of this paper is organized as follows: \S 2 is for preliminaries, detailed 
description and properties of the second stage Postnikov towers, constructions of the manifolds $\Sigma_q^5 /T$ and $X(q)^5$ are given; Theorem \ref{thm:one} is proved using the method of
modified surgery \cite{KreckM}, \cite{Kreck}, therefore in \S 3 we first determine the normal $2$-types and the corresponding bordism groups;
proof of the theorem is given in \S 4, where the major work is for the construction of the manifold $Y_1'$. 

\begin{acknowledgement}
I would like to thank the Hausdorff Institute for Mathematics  at Universit\"at Bonn for a research visit in December 2010. I am grateful  to  D.~Crowley, M.~Kreck and M.~Olbermann for helpful discussions. 
\end{acknowledgement}

\section{Preliminaries}\label{sec:two}

\subsection{Second space of the Postnikov tower}
Let $M^5$ be the quotient space of a free involution on $S^2 \times
S^3$, then $\pi_1(M) \cong \z/2$ and $\pi_2(M) \cong \z$. Accoring
to \cite{Baues}, the second space of the Postnikov tower of $M$,
$P_2(M)$, is a fibration over $\rp^{\infty}$ with fiber
$\cp^{\infty}$, and such a fibration is determined by its
$k$-invariant $k \in
H^3(\rp^{\infty};\underline{\pi_2(\cp^{\infty})})$, which is the
obstruction for a section.

If $\pi_2(M) \cong \z_+$ as a $\z[\z/2]$-module, since
$H^3(\rp^{\infty};\pi_2(\cp^{\infty}))=0$, we have $$P_2(M) =
\rp^{\infty} \times \cp^{\infty}.$$

If $\pi_2(M) \cong \z_-$, since
$H^3(\rp^{\infty};\underline{\pi_2(\cp^{\infty})}) \cong \z/2$,
there are two such fibrations up to fiber homotopy equivalence.
These can be described as follows:

Let $c$ denote the complex conjugation on $\cpi$, $(-1)$ denote the
antipodal map on $S^{\infty}$. Let $P = (\cpi \times
S^{\infty})/(c,-1)$, then there is a fibration $\cpi \to P \to \rpi$
with a section $\sigma \colon \rpi \to P$, $x \mapsto [x, \widetilde
x]$, where $\widetilde x$ is any preimage of $x$ in $S^{\infty}$.

There is a free involution $\tau$ on $\cpi$, with $\tau_*=(-1)$ on
$H_2(\cpi)$. (Under homogeneous coordinates, $\tau([z_0, z_1,
z_2, z_3, \cdots]) = [-\overline{z_1}, \overline{z_0}, -\overline{z_3}, \overline{z_2}, \cdots]$.) Let $Q =
(\cpi \times S^{\infty})/(\tau, -1)$, then there is a fibration
$\cpi \to Q \to \rpi$ which corresponds to the nontrivial
$k$-invariant. Since the involution $\tau$ on $\cpi$ is free, $Q$ is homotopy equivalent to $\cpi/\tau$, and the latter is an $\rp^2$-bundle over $\mathbb H \mathrm P^{\infty}$: $\rp^2 \to \cpi/\tau \to \mathbb H \mathrm P^{\infty}$. 
For details see \cite[p.~44]{Olbermann}

\smallskip

The homology groups of $Q$ were calculated by Olbermann
\cite{Olbermann}
\begin{lemma}\cite[p.~48--49]{Olbermann} \label{lemma:hmlg q}
\begin{enumerate}
\item $H^*(Q;\z/2) \cong \z/2[t,q]/t^3$, where $|t|=1$, $|q|=4$.
\item the integral homology of $Q$ up to dimension $6$ is given by
\begin{center}
\begin{tabular}{|c|c|c|c|c|c|c|}
\hline \vsp
                      & $1$    & $2$    & $3$    & $4$              & $5$         & $6$     \\ [.3ex]
\hline \vsp
          $H_*(Q;\z)$ & $\z/2$ & $0$ & $0$ & $ \z$ & $\z/2$  & $0$  \\[.3ex]
\hline
\end{tabular}
\end{center}
\item $H_4(Q;\z_-) \cong \z/2$, $H_5(Q;\z_-) = 0$, $H_6(Q;\z_-)\cong \z $.
\end{enumerate}
\end{lemma}

Analogously, we compute the homology groups of $P$:
\begin{lemma}\label{lemma:hmlg p} \
\begin{enumerate}
\item $H^*(P;\z/2) \cong \z/2[t,x]$, where $t$ is the pull-back of the nontrivial element of $H^1(\rpi;\z/2)$ and $x$
restricts to the nontrivial element of $H^2(\cpi;\z/2)$.
\item the integral homology of $P$ up to dimension $6$ is given by
\begin{center}
\begin{tabular}{|c|c|c|c|c|c|c|}
\hline \vsp
                      & $1$    & $2$    & $3$    & $4$              & $5$         & $6$     \\ [.3ex]
\hline \vsp
          $H_*(P;\z)$ & $\z/2$ & $\z/2$ & $\z/2$ & $\z \oplus \z/2$ & $(\z/2)^2$  & $(\z/2)^2$  \\[.3ex]
\hline
\end{tabular}
\end{center}
\item $H_5(P;\z_-) \cong \z/2$ is mapped injectively into $H_5(P;\z/2)$, whose image is dual to $t^3x$; $H_6(P;\z_-)\cong \z \oplus (\z/2)^2$.
\end{enumerate}
\end{lemma}
\begin{proof}
We apply the Serre spectral sequence for $H^*(-;\z/2)$
 to the fibration $\cpi \to P \to \rpi$. Since
there is a section $\sigma \colon \rpi \to P$, all $E^2_{p,0}$-terms
survive to infinity. By the multiplicative structure of the spectral
sequence, this implies that $H^*(M;\z/2) \cong \z/2[t,x]$. Using the
Serre spectral sequence for $H_*(-;\z)$ and the universal
coefficient theorem one computes $H_*(M;\z)$ in low dimensions.
Using the Bockstein sequence associated to the short exact sequences
$\z_+ \to \z[\z/2] \to \z_-$ and  $\z_- \to \z[\z/2] \to \z_+$, one
computes $H_*(M;\z_-)$ in low dimensions.

The long exact sequence associated to the coefficient sequence
$\z_- \stackrel{\cdot 2}{\rightarrow} \z_- \to \z/2$ shows that
$H_5(P;\z_-) \to H_5(P;\z/2)$ is injective. Let $Z_1$ be the
$5$-manifold defined in \S 1, we have a commutative diagram
$$\xymatrix{S^2 \ar@{^{(}->}[d] \ar[r] & Z_1 \ar[d]^{f_2} \ar[r] & \rp^3 \ar@{^{(}->}[d]\\
           \cpi \ar[r]       &  P   \ar[r]    & \rpi
           }$$
where $f_2 \colon Z_1 \to P$ is the second stage Postnikov map. By
comparison of the Serre spectral sequences of the two fibrations, we
see that $f_{2*} \colon H_5(Z_1;\z/2) \to H_5(P;\z/2)$ is injective
and the image is dual to $t^3x$. Therefore from the commutative
diagram
$$\xymatrix{H_5(Z_1;\z_-) \ar[d]^{f_{2*}} \ar[r] & H_5(Z_1;\z/2) \ar[d]^{f_{2*}}\\
           H_5(P;\z_-) \ar[r]                & H_5(P;\z/2)
           }$$
we conclude that the image of $H_5(P;\z_-) \to H_5(P;\z/2)$ is dual
to $t^3x$;
\end{proof}

If $M$ has $\pi_2(M) \cong \z_-$, then there is an exact sequence
(cf.~\cite{Brown})
$$H_3(\z/2) \to \z \otimes_{\z[\z/2]} \z_- \to H_2(M) \to H_2(\z/2).$$
Thus $H_2(M)$ is either trivial or isomorphic to $\z/2$. We see from
Lemma \ref{lemma:hmlg q} and Lemma \ref{lemma:hmlg p} that these two
cases correspond to different Postnikov towers.

\begin{corollary}\label{cor:postnikov}
Let $P_2(M)$ be the second space of the Postnikov tower of $M$. Then
$P_2(M) =P $ if and only if $H_2(M) \cong \z/2$; $P_2(M)=Q$ if and
only if $H_2(M)=0$.
\end{corollary}

\begin{lemma}\label{lemma:orientable}
Let $M^5$ be the quotient space of a free involution on $S^2 \times
S^3$ with $\pi_2(M) \cong \z_+$. Then $M$ is orientable.
\end{lemma}
\begin{proof}
Assume that $M$ is nonorientable, then as a $\z[\z/2]$-module,
$H_3(S^2 \times S^3) \cong \z_-$. An easy calculation with the
homology Serre spectral sequence with $\z$-coefficients for the
fibration $\w M \to M \to \rpi$ shows that $H_5(M;\z) \ne 0$, a
contradiction.
\end{proof}

\begin{lemma}\label{lemma:tors}
Let $M^5$ be a smooth $5$-manifold with $\pi_1(M) \cong \z/2$, $w_1(M)=0$ and $w_2(\w M)=0$. Then the torsion subgroup of $H_2(M;\z)$ has the form tors$H_2(M) \cong B \oplus B$, where $B$ is a finite abelian group.
\end{lemma}
\begin{proof}
Let $S^1 \hookrightarrow M$ be a representative of the nontrivial loop in $M$, $S^1 \times D^4 \subset M$ be its tubular neighborhood, $M_0=M-S^1 \times D^4$. Let $N^5$ be the result of surgery on this embedded $S^1$. Then it's easy to see that $\pi_1(M_0) \cong \z/2$, $\pi_1(N)=0$, $H_2(M) \cong H_2(M_0) \cong H_2(N)$, and $w_2(\w{M_0})=0$. 

There is a commutative diagram
$$\xymatrix{
\mathrm{Ext}(H_1(M_0), \z/2) \ar[r] & H^2(M_0; \z/2) \ar[r] & \mathrm{Hom}(H_2(M_0), \z/2) \ar[r] & 0 \\
 0 \ar[r] \ar[u] &  H^2(N; \z/2) \ar[r] \ar[u]^{i^*} & \mathrm{Hom}(H_2(N), \z/2) \ar[r] \ar[u]^{\cong} & 0 }$$
It's shown in \cite[Lemma 2.1]{Ham-Su} that $w_2(\w{M_0})=0$ implies $w_2(M_0) \in \mathrm{Ext}(H_1(M_0), \z/2)$. Note that $i^*(w_2(N))=w_2(M_0)$. This shows $w_2(N)=0$. Then by the classification of simply-connected spin $5$-manifolds \cite{Smale}, tors$H_2(N) \cong B \oplus B$.
\end{proof}

\begin{corollary}\label{coro:q-type}
Let $M^5$ be the quotient space of an orientation preserving free
involution on $S^2 \times S^3$ with $\pi_2(M) \cong \z_-$. Then
$H_2(M;\z)=0$. Therefore by Corollary \ref{cor:postnikov},
$P_2(M)=Q$.
\end{corollary}
\begin{proof}
It's seen from the Serre spectral sequence for the fibration $\w M \to M \to \rpi$ that $H_2(M)=0$ or $H_2(M) \cong \z/2$. But the latter is excluded by the above lemma.
\end{proof}

\begin{lemma}\label{lemma:spin}
Let $M^5$ be the quotient space of a free
involution on $S^2 \times S^3$ with $P_2(M)=P$. Then $M$ is
nonorientable and $w_2(M)=0$.
\end{lemma}
\begin{proof}
That $M$ is nonorientable is a consequence of the previous lemma. To
show that $w_2(M)=0$, consider the cohomology Serre spectral
sequence with $\z/2$-coefficients for the fibration $\w M \to M \to
\rpi$. $H^2(M;\z/2) \cong (\z/2)^2$ implies that the differentials
$d_3 \colon E^3_{p,2} \to E^3_{p+3,0}$  are trivial. Therefore
$H^2(M;\z/2)=\z/2 <t^2, x>$ and $H^3(M;\z/2)=\z/2 <t^3, tx>$, where
$t$ is the pull-back of the nontrivial element of $H^1(\rpi;\z/2)$
and $x$ restricts to the nontrivial element of $H^2(\w M;\z/2)$. By
Poincar\'e duality, $H^5(M;\z/2)$ is generated by $t^3x$.

Let $v=v_1+v_2$ denote the Wu class of $M$. Then $v_1=w_1(M)=t$,
$w_2(M)=(Sq v)_2=v_1^2+v_2=t^2+v_2$, where $v_2$ is determined by
$Sq^2 \colon H^3(M;\z/2) \to H^5(M;\z/2)$. We have $Sq^2 t^3=t^5=0$
and $Sq^2 tx = t^2Sq^1x$.

Let $f_2 \colon M \to P=P_2(M)$ be the Postnikov map, there is a
commutative diagram
$$\xymatrix{H^2(P;\z/2) \ar[d]^{Sq^1} \ar[r]^{f_2^*} & H^2(M;\z/2) \ar[d]^{Sq^1}\\
           H^3(P;\z/2) \ar[r]^{f_2^*}                 & H^3(M;\z/2)
           }$$
by Lemma \ref{lemma:hmlg p} and the structure of $H^3(M;\z/2)$, both
$f_2^*$ are isomorphsims. Therefore we only need to determine
$Sq^1x$ for a special $M$. Consider the manifold $Z_1=S(\eta_3
\oplus 2\underline{\mathbb R})$. It is easy to compute that
$H_2(Z_1)=\z/2$ and $w_2(Z_1)=0$. This implies that $Sq^1x =tx$ and
$v_2=t^2$. Hence $w_2(M)=0$.
\end{proof}

\subsection{Characteristic submanifolds and $\Pin^{\pm}$-structures}
Recall that for a manifold $M^n$ with fundamental group $\z /2$, a
\emph{characteristic submanifold} $N^{n-1} \subset M$ is defined in
the following way (see \cite[\S 5]{Thomas-Geiges}): there is a
decomposition $\w{M}=A \cup TA$ such that $\partial A =\partial TA =
\w{N}$, where $T$ is the deck-transformation. Then $N:= \w{N}/T$ is
called the characteristic submanifold of $M$. For example, if
$M=\rp^n$, then $N=\rp^{n-1}$. In general, let $f\colon M \to \rp^k$
($k$ large) be the classifying map of the universal cover,
transverse to $\rp^{k-1}$, then $N$ can be taken as
$f^{-1}(\rp^{k-1})$. By equivariant surgery we may assume that
$\pi_1(N) \cong \z/2$ and that the inclusion $i\colon  N \subset M$
induces an isomorphism on $\pi_1$.

\smallskip

Recall that there are central extensions of $O(n)$ by $\z /2$ (see
\cite[\S 1]{Kirby-Taylor} and \cite[\S 2]{HKT}):
$$1 \to \mathbb Z/2 \to \Pin^{\pm}(n) \to O(n) \to 1,$$

Let $\dagger \in \{+, - \}$. A $\Pin^{\dagger}$-structure on a
vector bundle $\xi$ over a space $X$ is the fiber homotopy class of
lifts of the classifying map $c_E\colon  X \to BO$.

\begin{lemma}\cite[Lemma 1]{HKT}
\
\begin{enumerate}
\item
A vector bundle $E$ over $X$ admits a $\Pin^{\dagger}$-structure if
and only if
$$\begin{array}{ll}
w _2(E)=0 & \mathrm{for}\ \ \dagger =+, \\
w_ 2(E)=w_ 1(E)^2 & \mathrm{for}\ \ \dagger =-, \\
\end{array}$$

\item
$\Pin^{\pm}$-structures are in bijection with $H^1(X;\z/2)$.
\end{enumerate}
\end{lemma}

\begin{lemma}\cite[Lemma 9]{Thomas-Geiges}\label{lemma:thom-geig}
Let $M$ be a smooth, orientable $5$-manifold with $\pi_1(M) \cong
\z/2$. Let $N \subset M$ be the characteristic submanifold (with
$\pi_ 1(N)\cong \z/2$). Then $TN$ admits a
$\Pin^{\dagger}$-structure, where
$$\dagger = \left \{ \begin{array}{ll}
- & \textrm{if}\ \ w_2(M)=0  \\
  & \\
+ & \textrm{if}\ \ w_2(M)\ne 0  \ \ \textrm{and} \ \ w_2(\w{M})=0 
\end{array} \right.$$
\end{lemma}

$N$ admits a pair of $\Pin^{\dagger}$-structures. The
$\Pin^{\dagger}$-bordism class of the pair of structures is
independent of the choice of characteristic submanifold $N$. The
$\Pin^{\dagger}$-bordism groups in low dimensions are computed by
Kirby and Taylor in \cite{Kirby-Taylor}, we have $\Omega_4^{\Pin^+}
\cong \z/16$ generated by $\pm \rp^4$ and $\Omega_4^{\Pin^-}=0$.

\subsection{Construction of manifolds}
In this subsection we give detailed description of the manifolds $\Sigma_q^5/T$ and $X(q)^5$.

The manifolds $\Sig_q/T$  are described in \cite[\S
1]{Thomas-Geiges}: Let $V^6_q$ be the complex hypersurface in
$\mathbb C^4$ given by the equation
$$z_0^q+z_1^2+z_2^2+z_3^2=0,$$
where $q=0,1,\dots, 8$. The origin is an isolated singularity of
$V^6_q$, and the link $\Sig_q$ of this singularity, i.~e., the
intersection of $V^6_q$ with the unit sphere $S^7\subset \mathbb
C^4$, is shown to be diffeomorphic to $S^5$ for $q$ odd and
diffeomorphic to $S^2 \times S^3$ for $q$ even. The involution $T
\colon \Sig_q \to \Sig_q$ given by
$$T(z_0,z_1,z_2,z_3)=(z_0,-z_1,-z_2,-z_3)$$
is free on $\Sig_q$.

Thus for $q=0,2,4,6,8$, we have an orientation preserving free
involution on $S^2 \times S^3$, we denote the  quotient space by
$\Sig_q /T$. 
\begin{lemma}
$H_2(\Sig_q/T)=0$, $w_2(\Sig_q/T) \ne 0$,
the $\Pin^+$-bordism class of the charateteristic submanifold of
$\Sig_q/T$ is $\pm q \in \Omega_4^{\Pin^+} \cong \z/16$. $P_2(\Sig_q/T)=Q$.
\end{lemma}
\begin{proof}The first statement was shown in \cite[Proposition 6, Lemma
11]{Thomas-Geiges}. Since $H_2(\Sig_q/T)=0$, by
Corollary \ref{cor:postnikov} $P_2(\Sig_q/T)=Q$.
\end{proof}

For $q=1,3,5,7$, the above construction gives a free involution on
$S^5$, the quotient space $\Sig_q /T$ is a fake $\rp^5$ --- a smooth
manifold homotopy equivalent to $\rp^5$ (see \cite[\S 14D]{Wall}).
We denote it by $H\rp^5_q$. Furthermore, it is shown that these
$H\rp^5_q$ are all the possible fake $\rp^5$. The $\Pin^+$-bordism
class of the corresponding characteristic submanifold is $\pm q \in
\Omega_ 4^{\Pin^+}$.

\smallskip

The manifolds $\XX{q}$ are constructed by the ``connected-sum along
$S^1$'' operation on the fake projective spaces $H\rp^5_q$ (see
\cite[\S 3]{Ham-Su}). Denote the trivially oriented $4$-dimensional
real disc bundle over $S^1$ by $E$. Choose embeddings of $E$ into
$\HH{q}$ and $\HH{q'}$, representing a
 generator of $\pi _1$, such that the first embedding preserves
  the orientation and the second reverses it. Then we define
$$\HH{q}\scs \HH{q'} := (\HH{q}-E)\cup _{\partial} (\HH{q'}-E),$$
 Note that $\HH{q}$ admits orientation reversing
automorphisms, thus the construction doesn't depend on the
orientations. (The fact that $H\rp^5_q$ admits orientation reversing
automorphisms follows from that $\rp ^5$ admits orientation
reversing automorphisms and that the action of $Aut(\rp^5)$ on the
structure set $\mathscr{S}(\rp^5)$ is trivial.)

The Seifert-van Kampen theorem implies that $\pi_ 1(\HH{q} \scs
\HH{q'}) \cong \z /2$. The Mayer-Vietoris exact sequence implies
that $$H_ 2(\HH{q} \scs \HH{q'})\cong \z, \ \ \  H_ 2(\HH{q} \scs
\HH{q'};\z[\z/2])\cong \z.$$

Since $\pi_ 1 SO(4) \cong \mathbb Z/2$, there are actually two
possibilities to form $\HH{q} \scs \HH{q'}$. However, note that the
characteristic submanifold of $\HH{q}\scs \HH{q'}$ is $N_ 1 \scs N_
2$, where $N_1$ (resp.~$N_2$) is the characteristic submanifold of
$\HH{q}$ (resp.~$\HH{q'}$). (See \cite[p.651]{HKT} for the
definition of $\scs$ for nonorientable $4$-manifolds with
fundamental group $\z /2$). Therefore if we fix $\Pin^+$-structures
on each of the characteristic submanifolds, then the manifold $H\rp^5_q \scs
H\rp^5_{q'}$ is well-defined.

This construction allows us to construct manifolds with given
bordism class of characteristic submanifold. Note that $N _1\scs N_
2$ corresponds to the addition in the bordism group $\Omega_
4^{\Pin^{\dagger}}$. Now for $q=0,2,4,6,8$, choose $l,l' \in
\{1,3,5,7\}$ and appropriate $\Pin^+$-structures on the
characteristic submanifolds of $H\rp^5_{l}$ and $H\rp^5_{l'}$, we
can form a manifold $H\rp^5_{l} \scs H\rp^5_{l'}$ such that the
$\Pin^+$-bordism class of the  characteristic submanifold is $ \pm q
\in \Omega_ 4^{\Pin^+}$. We denote this manifold by $\XX{q}$. For
example, we can form $\XX{0}=H\rp^5_{1} \scs H\rp^5_{1}$ and
$\XX{2}=H\rp^5_{1} \scs H\rp^5_{1}$ with different glueing maps.

The manifold $\XX{q}$ has fundamental group $\z/2$ and $w_2(X^5(q))
\ne 0$. The universal cover $\w{\XX{q}}$ is a simply-connected
$5$-manifold with trivial second Stiefel-Whitney class and
$H_2(\w{\XX{q}}) \cong \z$. Therefore, by the classification theorem
of Smale \cite{Smale}, $\w{\XX{q}} \cong S^2 \times S^3$. The action
of $\pi_1(X^5(q))$ on $\pi_2(X^5(q))$ is trivial
(cf.~\cite[p.~6--7]{Ham-Su}).

\smallskip

The manifolds $Y_i$ and $Z_j$ are sphere bundles over projective
spaces. Their homology groups and characteristic classes can be
computed by standard methods. Especially we have 
\begin{itemize}
\item $H_2(Y_1)=0$,
$w_1(Y_1)=w_2(Y_1)=0$; 
\item $H_2(Y_2)=0$, $w_1(Y_2) \ne 0$, $w_2(Y_2)=0$;
\item $H_2(Z_1)\cong \z/2$, $w_1(Z_1)\ne 0$, $w_2(Z_1)=0$.
\end{itemize}

\section{Normal $2$-types and Bordism Groups}
\subsection{Modified surgery}
The classification of free involutions on a manifold is equivalent
to the classification of the quotient manifolds. For involutions on
the spheres $S^n$, the quotient spaces are all homotopy equivalent
to $\rp^n$. Therefore the classical surgery program can be applied
to the classification problem of the quotient spaces (\cite[\S
14D]{Wall}).

For free involutions on $S^2 \times S^3$, the homotopy type of the
quotient spaces is not constant. However, the lower Postnikov
system and the normal data of the quotient spaces are relatively
simple. This is appropriate for the settings of the modified surgery
program in \cite{KreckM}, \cite{Kreck}. In this subsection we briefly describe the
modified surgery program. For the sake of simplicity, we only
consider $5$-dimensional manifolds here.

\smallskip

Let $p\colon  B \to BO$ be a fibration. A normal $B$-structure of
$M$ is a lift $\bar{\nu}\colon M^5 \to B$ of the normal Gauss map
$\nu\colon  M \to BO$. $\bar{\nu}$ is called a normal $2$-smoothing
if it is a $3$-equivalence. Manifolds with normal $B$-structures
form a bordism theory. Suppose $(M^5_ i, \bar{\nu_ i})$ ($i=1,2$)
are two normal $2$-smoothings in $B$, $(W^6, \bar{\nu})$ is a
$B$-bordism between $(M^5_ 1, \bar{\nu_ 1})$ and $(M^5_ 2, \bar{\nu
_2})$. Then $W$ is bordant rel.~boundary to an $s$-cobordism
(implying that $M_ 1$ and $M_ 2$ are diffeomorphic) if and only if
an obstruction $\theta (W, \bar{\nu}) \in L _6(\pi _1(B), w_1(B))$
is zero (\cite{KreckM} and \cite[p.730]{Kreck}).

The obstruction group $L_ 6(\pi_1(B), w_1(B))$ is related to the
ordinary Wall's $L$-group in the following exact sequence
$$0 \to L_ 6^s(\pi_ 1(B), w_1(B)) \to L _6(\pi_ 1(B), w_1(B)) \to \Wh(\pi_ 1(B)),$$
where $L_ 6^s(\pi_ 1(B), w_1(B))$ is the Wall's $L$-group and
$\Wh(\pi_ 1(B))$ is the Whitehead group of $\pi_ 1(B)$. For $\pi_
1(B)=\z/2$, $\Wh(\z/2)=0$ (\cite{Milnor}), therefore we have an
isomorphism $L_ 6^s(\pi_ 1(B), w_1(B)) \cong L_ 6(\pi_ 1(B),
w_1(B))$. According to \cite[\S 13A]{Wall},  $L_
6^s((\z/2)^{\pm}) \cong \z/2$, detected by the Kervaire-Arf invariant. And
if $\theta (W, \bar{\nu})$ is nonzero, then one can do surgery on
$(W, \bar{\nu})$ such that the result manifold $(W', \bar{\nu}')$
has trivial surgery obstruction. Therefore we have the following

\begin{proposition}\label{prop:one}
Two smooth $5$-manifolds $M_1$ and $M_2$ with fundamental group
$\z/2$ are diffeomorphic if they have bordant normal $2$-smoothings
in some fibration $B$.
\end{proposition}

The fibration $B$ is called the normal $2$-type of $M$ if $p$ is
$3$-coconnected. This is an invariant of $M$. By this proposition,
the solution to the classification problem consists of two steps:
first we need to determine the normal $2$-types $B$ for the
$5$-manifolds under consideration. The second step is to compute the
corresponding bordism groups $\Omega _5^{(B,p)}$.

\subsection{Normal $2$-types}
In this subsection we determine the normal $2$-types of the quotient
manifolds of smooth free involutions on $S^2 \times S^3$.

\smallskip

Let $p_2 \colon B\Spin \to BO$ be the canonical projection, $\oplus
\colon BO \times BO \to BO$ be the $H$-space structure on $BO$
induced by the Whitney sum of vector bundles. Let $\eta$ be the
canonical real line bundle over $\rp^{\infty}$. Let $\pi \colon
P_2(M) \to \rpi$ be the projection in the Postnikov tower, $L=\pi^*
\eta$ be the pull-back line bundle and $kL$ be the Whitney sum of
$k$ copies of $L$. Let $f_2 \colon M \to P_2(M)$ be the second stage
Postnikov map.

\smallskip

Consider the fibration
$$p\colon  B=P_2(M) \times B\Spin \stackrel{p _1 \times p_
2}{\longrightarrow} BO \times BO \stackrel{\oplus}{\longrightarrow}
BO,$$ where $p_ 1\colon  P_2(M) \to BO$ is the classifying map of
$kL$ with
$$k = \left \{ \begin{array}{ll}
0 & \textrm{if}\ \ w_1(M)=w_2(M)=0  \\
1 & \textrm{if}\ \ w_1(M) \ne 0, \ \  w_2(M) \ne 0 \\
2 & \textrm{if}\ \ w_1(M)=0, \ \ w_2(M) \ne 0 \\
3 & \textrm{if}\ \ w_1(M) \ne 0, \ \ w_2(M)=0 \\
\end{array} \right.$$

A straightforward calculation shows that $w_1(\nu M
-f_2^*kL)=w_2(\nu M -f_2^*kL)=0$. This implies that $\nu M -
f_2^*kL$ admits a $\Spin$-structure. Such a structure induces a map
$M \to B\Spin$. Together with $f_2$ we have a lift $\bar{\nu}$ of
$\nu$. It is easy to see that $(B,p)$ is the normal $2$-type of $M$
and $\bar{\nu}$ is a normal $2$-smoothing. (Compare \cite[\S 5A]{Ham-Su})

\subsection{Computation of $\Omega_5^{(B,p)}$}
To apply Proposition \ref{prop:one}, according to the normal
$2$-types given above and Lemma \ref{lemma:orientable},
Corollary \ref{coro:q-type}, Lemma \ref{lemma:spin}, we need to compute the
following bordism groups $\Omega_5^{(B,p)}$:
\begin{itemize}
\item $\OS_5(\rpi \times \cpi;kL)$, $k=0,2$;

\

\item $\OS_5(Q;kL)$,  $k=0,1,2,3$;

\

\item $\OS_5(P;3L)$.
\end{itemize}

\

\noindent \underline{Computation of $\OS_5(Q;kL)$:}

\medskip

We apply the Atiyah-Hirzebruch spectral sequence and (implicitly) the Adams
spectral sequence to compute the bordism groups $\OS_5(Q;kL)\cong
\widetilde{\Omega}^{\Spin}_{5+k}(\mathrm{Th}(kL))$.

For $k=1,3$, the Atiyah-Hirzebruch spectral sequence has
$E^2_{p,q}=H_p(Q;\underline{\OS_q})$ (here we use the Thom
isomorphism to identify $\w H_{p+k}(\mathrm{Th}(kL);\OS_q)$ and
$H_p(Q;\underline{\OS_q})$). Using Lemma \ref{lemma:hmlg q} the
relevant terms and differentials in the spectral sequence are
depicted as follows:
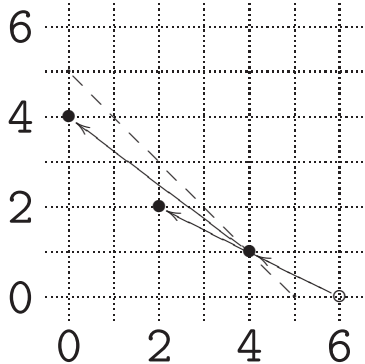
\begin{center}
\def\sseqgridstyle{\ssgridgo}
\sseqentrysize=.6cm
\def\sseqpacking{\sspackhorizontal}
\begin{sseq}{7}{7}
\ssmoveto 0 5  \ssdrop{\cdot} \ssdashedline {5}{-5}
 \ssdrop{\cdot}

\ssmoveto 6 0 \ssdrop{\circ} \ssarrow {-2}{1} \ssdrop{\bullet}

 \ssarrow {-2}{1} \ssdrop{\bullet}

 \ssmoveto 4 1 \ssarrow {-4} {3} \ssdrop{\bullet}
\end{sseq}
\end{center}

Each black dot denotes a copy of $\z/2$ and the circle denotes a
copy of $\z$.

\begin{enumerate}
\item
$k=1$: the differential $d_2 \colon E^2_{4,1} \to E^2_{2,2}$ is
dual to $Sq^2+t\cdot Sq^1$, $d_2 \colon E^2_{6,0} \to E^2_{4,1}$ is
reduction mod $2$ composed with the dual of $Sq^2+t\cdot Sq^1$ (cf.~
\cite{Teichner}). From Lemma \ref{lemma:hmlg q}, both are trivial.
Hence $\OS_5(Q;L)$ is either trivial or isomorphic to $\z/2$. By
comparison with the Adams spectral sequence for
$\OS_5(Q;L)=\pi_5(TL\wedge M\Spin)$, Olbermann showed in
\cite{Olbermann} that $\OS_5(Q;L)=0$.

\item
$k=3$: the differential $d_2 \colon E^2_{6,0} \to E^2_{4,1}$ is
reduction mod $2$ composed with the dual of $Sq^2+t\cdot
Sq^1+t^2\cdot$. It is shown in \cite{Olbermann} that $Sq^1
q=Sq^2q=0$ (using the fact the $Q \simeq \cp^{\infty}/\tau$ and
there is a fiber bundle $\rp^2 \to \cp^{\infty}/\tau \to \mathbb H
\mathrm P^{\infty}$). Therefore $Sq^2q+t\cdot Sq^1q+t^2\cdot q=t^2q$
is the nontrivial element in $H^6(Q;\z/2)$, and hence $d_2$ is
surjective. Therefore $\OS_5(Q;3L)=0$.
\end{enumerate}

For $k=0,2$, the Atiyah-Hirzebruch spectral sequence has
$E^2_{p,q}=H_p(Q; \OS_q)$. Using Lemma \ref{lemma:hmlg q} the
relevant terms and differentials of the spectral sequence are
depicted as follows:

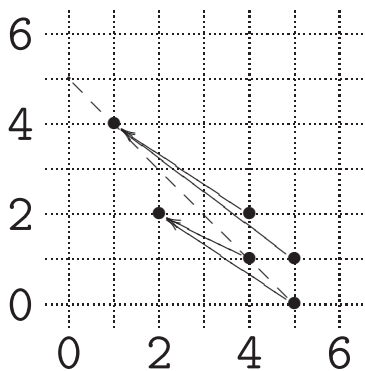
\begin{center}
\def\sseqgridstyle{\ssgridgo}
\sseqentrysize=.6cm
\def\sseqpacking{\sspackhorizontal}
\begin{sseq}{7}{7}
\ssmoveto 0 5  \ssdrop{\cdot} \ssdashedline {5}{-5}
  \ssdrop{\bullet} \ssarrow {-3}{2} \ssdrop{\bullet}

 \ssmoveto 4 1 \ssdrop{\bullet} \ssarrow {-2}{1}

 \ssmoveto 5 1 \ssdrop{\bullet} \ssarrow {-4}{3} \ssdrop{\bullet}

 \ssmoveto 4 2 \ssdrop{\bullet} \ssarrow {-3}{2} \ssmoveto 0 5
\end{sseq}
\end{center}

Each black dot denotes a copy of $\z/2$. $d_2 \colon E^2_{4,1} \to
E^2_{2,2}$ is dual to $Sq^2$ for $k=0$ and to $Sq^2+t^2 \cdot$ for
$k=2$. In both cases $d_2$ is trivial.

For $k=2$, by Lemma \ref{lemma:thom-geig} there is a homomorphism
(see also \cite[p.~18]{Ham-Su})
$$\phi \colon \OS_5(Q;2L) \to \Omega_4^{\Pin^+}$$
sending a singular manifold $[X,f]$ to the bordism class of its
characteristic submanifold. We know that $\Omega_4^{\Pin^+}$ is
isomorphism to $\z/16$ and the characteristic submanifold of
$\Sig_q/T$ ($q=0,2,4,6,8$) corresponds to $\pm q \in
\Omega_4^{\Pin^+}$. It is seen from the spectral sequence that
$\OS_5(Q;2L)$ has at most $8$ elements. Therefore $\phi$ is
injective and $\OS_5(Q;2L) \cong \z/8$, determined by the
characteristic submanifold. The action of the group of fiber homotopy equivalences $Aut(B \stackrel{p}{\to} BO)$ on $\OS_{5}(Q; 2L)$ has the effect of multiplication by $\pm 1$. This can be seen from the action of $Aut(B \stackrel{p}{\to} BO)$ on the $E^{2}$-page of the Atiyah-Hirzebruch spectral sequence.

The $k=0$ case is more subtle. First consider the long exact sequence for the pair $(DL, SL)$, where $DL$ is the disk bundle of $L$ and $SL$ the sphere bundle:
$$\OS_6(DL, SL) \to \OS_5(SL) \to \OS_5(DL) \to \OS_5(DL, SL) \to \OS_4(SL).$$
Note that we have $SL \simeq \cpi$, $DL \simeq Q$ and $\OS_*(DL, SL) \cong \OS_{*-1}(Q;L)$ by Thom isomorphism. From \cite{Olbermann} we know that $\OS_5(Q;L)=0$ and $\OS_4(Q;L)\cong \z/2$. An easy calculation by the Atiyah-Hirzebruch spectral sequence gives $\OS_4(\cpi) \cong \z \oplus \z$. Therefore we have $\OS_5(Q) \cong \z/2$. By looking at the Atiyah-Hirzebruch spectral sequence depicted above we know that the nontrivial element comes from position $(4,1)$. Also note that since $\OS_{5}(Q) \cong \z/2$, the action of the group of fiber homotopy equivalences $Aut(\mathrm{BSpin} \times Q \to \mathrm{BO})$  on $\OS_{5}(Q)$ is trivial. 

A generator for this bordism group is given as follows: first note that there is an embedding  
$$\cp^{1} \times \cp^{1} \hookrightarrow \cp^{3}, \ \ ([z_{0}, z_{1}], [w_{0}, w_{1}]) \mapsto [z_{0}w_{0}, z_{1}w_{1}, z_{0}w_{1}, z_{1}w_{0}]$$
which is compatible with the involutions $([z_{0}, z_{1}], [w_{0}, w_{1}]) \mapsto ([\overline{z_{1}}, \overline{z_{0}}], [-\overline{w_{1}}, \overline{w_{0}}])$ on $\cp^{1} \times \cp^{1}$ and $[\zeta_{0}, \zeta_{1}, \zeta_{2}, \zeta_{3} ] \mapsto [-\overline{\zeta_{1}}, \overline{\zeta_{0}}, -\overline{\zeta_{3}}, \overline{\zeta_{2}} ]$ on $\cp^{3}$, and the image of this embedding is a complex surface of  
degree $2$
$$V(2) \colon \zeta_0 \zeta_1 - \zeta_2 \zeta_3=0.$$ 
This gives a singular manifold
$$f \colon V(2)/\tau \to \cp^{3}/\tau \to \cpi/\tau \simeq Q$$
Since $i_* \colon H_4(\cpi) \to H_4(Q)$ is an isomorphism, $f_* \colon H_4(V(2)/\tau) \to H_4(\cpi/\tau)$ is also an isomorphism. $V(2)/\tau$ is the sphere bundle of the vector bundle $\eta \oplus 2 \underline {\mathbb R}$ over $\rp^{2}$, which is spin (c.~f.~\cite[Remark 4.5]{Ham-Kr}. This implies that the edge homomorphism $\OS_4(Q) \to H_4(Q)$ is an surjection. Since the Atiyah-Hirzebruch spectral sequence is a module over the coefficient ring $\OS_*$, we see that the nontrivial element in $\OS_5(Q)$ is represented by 
$$V(2)/\tau \times S^1 \stackrel{pr_1}{\longrightarrow} V(2)/\tau \to \cpi/\tau \simeq Q,$$ 
where $S^1$ is given the nontrivial spin structure. 

A bordism invariant detecting the nontrivial element is given as follows: first it's easy to see that the map $Q \simeq \cp^{\infty}/\tau \to
\mathbb H \mathrm P ^{\infty}$ induces an isomorphism $\OS_5(Q)
\stackrel{\cong}{\longrightarrow} \OS_5(\mathbb H \mathrm P
^{\infty})=\OS_5(S^4)$, which is given as follows: for a singular manifold $f \colon X^{5} \to Q$, by cellular approximation, we can find a map $g \colon X^{5} \to S^{4}$ such that the following diagram commutes up to homotopy
$$\xymatrix{
X \ar[d]_{g} \ar[r]^{f} & Q \ar[d]^{pr} \\
S^{4} \ar[r]^{i} & \mathbb H \mathrm P^{\infty}\\
}$$
Then the  bordism class $[X,g]\in \OS_{5}(S^{4})$ is the image of $[X, f] \in \OS_{5}(Q)$. 
Let $s_0 \in S^4$ be a  point, assuming that $g \colon X \to S^{4}$ is  transverse to $s_0$, then $g^{-1}(s_0)$ is a
$1$-manifold with spin structure, whose bordism class in $\OS_{1} \cong \z/2$ is an invariant of $[X,f] \in \OS_{5}(Q)$.  Comparing with the definition of a characteristic submanifold in \S 2B, we call $g^{-1}(s_{0})$  the $1$-dimensional characteristic submanifold of $[X,f]$.

\

\noindent \underline{Computation of $\OS_5(P;3L)$:}

\medskip

The $E^2$-terms of the Atiyah-Hirzebruch spectral sequence are
$E^2_{p.q}=H_p(P; \underline{\OS_q})$, the differential $d_2 \colon
E^2_{p,1} \to E^2_{p-2,2}$ is dual to $Sq^2+t\cdot Sq^1 +t^2\cdot$,
$d_2 \colon E^2_{p,0} \to E^2_{p-2,1}$ is reduction mod $2$ composed
with the dual of $Sq^2+t\cdot Sq^1+t^2\cdot$. The image of the mod
$2$ reduction $H_6(P;\z_-) \to H_6(P;\z/2)$ contains an element $a$
such that $\langle t^6, a \rangle =1$ and an element $b$ such that
$\langle t^2x^2, b \rangle =1$. (Using similar method in the proof
of Lemma \ref{lemma:hmlg p}, $a$ is obtain from the section $\sigma
\colon \rpi \to P$ and $b$ is obtained from the commutative diagram
of fibrations
$$\xymatrix{\cp^2 \ar[d] \ar[r] & V^6 \ar[d] \ar[r] & \rp^2 \ar[d]\\
           \cpi \ar[r]       &  P   \ar[r]    & \rpi
           }$$
where $V^6 = (\cp^2 \times S^2)/(c,-1)$.)

Having this information, a standard calculation shows that all terms
in the line $p+q=5$ are killed except for $E^2_{5,0} \cong
H_5(P;\z_-)$. Therefore the edge homomorphism $\OS_5(P;3L) \to
H_5(P;\z_-)$ is an isomorphism, $\OS_5(P;3L) \cong \z/2$, and a
bordism class $[X,f]$ is determined by the number $\langle t^3x,
f_*[X]_{\z/2} \rangle$.

\medskip

The groups $\OS_5(\rpi \times \cpi;kL)$, $k=0,2$ were computed in
\cite[Propostion 5.3, Proposition 5.6]{Ham-Su}.

\smallskip

We summarize the above computations in:

\begin{proposition} \label{prop:two} \
\begin{enumerate}
\item
$\Omega_ 5^{\Spin}(\rp^{\infty} \times \cp^{\infty}) \cong \z/4$. A
generating bordism class $[X^5,f]$ is detected by the invariant
$\langle  t^3 x, f_ *[X] \rangle \in \z/2$, the two generating
bordism classes are interchanged if we compose $f$ with the
involution $\tau$ on $\cp^{\infty}$.

\

\item There is a short exact sequence
$$0 \to \z/4 \to \Omega _5^{\Spin}(\rp^{\infty}
\times \cp^{\infty};2L) \to \Omega _4^{\Pin^+} \to 0,$$ the bordism
classes $[X,f] \in \{\pm 1\} \subset \z/4$ are detected by the
invariant $\langle t^3 x , f_
*[X] \rangle \in \z/2$, and they are interchanged if we compose $f$ with the
involution $\tau$ on $\cp^{\infty}$.

\

\item
$\OS_5(Q) \cong \z/2$, detected by the Spin-bordism class of the $1$-dimensional characteristic submanifold. The action of $Aut(\mathrm{BSpin} \times Q \to \mathrm{BO})$ on $\OS_{5}(Q)$ is trivial.

\

\item 
$\OS_5(Q;L) =
\OS_5(Q;3L)=0$. 

\

\item
 $\OS_2(Q;2L) \cong \z/8$ is determined by the
characteristic submanifold. The action of $Aut(B \stackrel{p}{\to} \mathrm{BO})$ on $\OS_{5}(Q;2L)$ is multiplication by $\pm 1$.

\

\item
$\OS_5(P;3L) \cong \z/2$, a bordism class $[X,f]$ is determined by
the bordism number $\langle t^3x, f_*[X]_{\z/2} \rangle$.
\end{enumerate}
\end{proposition}

\section{Proof of the theorem}

Now we are ready to prove Theorem \ref{thm:one}. This is a direct consequence of the classification machinery 
Proposition \ref{prop:one} and the computation results in Proposition\ref{prop:two}, except for the $(P_2(M)=Q, w_{1}(M)=0, w_{2}(M)=0)$-case, which needs more careful analysis.

\begin{proof}[Proof of Theorem \ref{thm:one}]
The $\pi_2(M) \cong Z_+$ case was shown in \cite[Theorem 3.1,
Theorem 3.6]{Ham-Su}. Here we only need to consider the $\pi_2(M)
\cong \z_-$ cases.

For $w_1(M) \ne 0$ and $P_2(M)=Q$, since the corresponding bordism group
$\Omega_5^{(B,p)}=0$, there is only one diffeomorphism type, which is given in the table: $Y_2$ and $S^3 \times \rp^2$ respectively.

For $w_1(M) \ne 0$ and $P_2(M)=P$, the bordism group is
$\Omega_5^{(B,p)}=\OS_5(P;3L) \cong \z/2$. For a normal
$2$-smoothing $[M, \bar{\nu}]$, the bordism number $\langle t^3x,
\bar{\nu}_*[M] \rangle $ must be $1$. Therefore there is only
one diffeomorphism type, which is represented by $Z_1$.

For $w_1(M)=0$ (and therefore $P_2(M)=Q$), if $w_2(M) \ne 0$, the corresponding bordism group
$\Omega_5^{(B,p)}=\OS_5(Q;2L) \cong \z/8$. Each bordism class $q$ is
represented by a normal $2$-smoothing $[\Sig_q/T, \bar{\nu}]$, where $q$ and $-q$ are interchanged by the action of $Aut(B \stackrel{p}{\to} \mathrm{BO})$.
Therefore  $M$  is diffeomorphic to a unique $\Sig_q/T$, $q= 0,2,4,6,8$.
 
 \medskip
 
Now we turn to the case where $P_2(M)=Q$, $w_{1}(M)=0$, and $w_{2}(M)=0$. The bordism group is $\OS_{5}(Q) \cong \z/2$. Our aim is to show each bordism class is represented by a normal $2$-smoothing, therefore there are two diffeomorphism classes in this case.  For this we need the following two lemmas.

\begin{lemma} \label{lem:3.3}
$[Y_{1}, f]=0 \in \OS_{5}(Q)$, where $f \colon Y_{1} \to Q$ is the second stage Postnikov map. 
\end{lemma}
\begin{proof}
$Y_{1}$ is an $S^{3}$-bundle over $\rp^{2}$, and the projection $p \colon Y_{1} \to \rp^{2}$ induces isomorphisms on $\pi_{1}$ and $\pi_{2}$. $Q \simeq \cpi /\tau$, which is an $\rp^{2}$-bundle over $\mathbb H \mathrm P^{\infty}$, the inclusion $i \colon \rp^{2} \to Q$ induces isomorphisms $\pi_{1}$ and $\pi_{2}$. Therefore $f \colon Y_{1} \to Q$ can be taken as the composite $i \circ p$. Recall from Proposition \ref{prop:two} that the bordism class $[Y_{1}, f]$ is determined by its $1$-dimensional characteristic manifold. Now since $f$ factors through $\rp^{2}$, the transversal preimage of $s_{0} \in S^{4}$ is empty. Therefore  $[Y_{1}, f]=0 \in \OS_{5}(Q)$. 
\end{proof}

\begin{lemma}
There is a normal $2$-smoothing $Y_{1}' \to \mathrm{BSpin} \times Q$ representing the nontrivial element in $\OS_{5}(Q)$.
\end{lemma}
\begin{proof}
First recall that the nontrivial element in $\OS_{5}(Q)$ is represented by 
$$V(2)/\tau \times S^{1} \stackrel{pr}{\longrightarrow} V(2)/\tau \to \cp^{3}/\tau \to \cpi/\tau \simeq Q.$$
Denote $V(2)/\tau$ by $N^{4}$. We have $\pi_{1}(N) \cong \z/2$, mapping isomorphically to $\pi_{1}(Q)$;  under a canonical basis, the map $\pi_{2}(N) \cong \z_{-}^2 \to \pi_{2}(Q) \cong \z_{-}$ is represented by the matrix $(1,1)$. 

Since $\pi_{1}(N)$ is mapped isomorphically to $\pi_{1}(Q)$ and the factor $S^{1}$ is mapped trivially to $Q$, we may first do framed $1$-surgery on $S^{1}$ to obtain $g \colon M^{5} \to Q$ such that $g_{*}$ is an isomorphism on $\pi_{1}$.  Now let us compute  $\pi_{2}(M)$ and $g_{*}$ on $\pi_{2}$. Now $M^5$ has the form $M=(N \times S^{1} - S^{1} \times D^{4}) \cup_{S^{1} \times S^{3}} D^{2} \times S^{3}$. Denote $N \times S^{1} - S^{1} \times D^{4}$ by $M_{0}$.  Let $\Lambda=\z[\z/2]$ be the group ring, the Mayer-Vietoris sequence with $\Lambda$-coefficients is 
$$0 \to H_{2}(M_{0}; \Lambda) \to H_{2}(M;\Lambda) \to H_{1}(S^{1} \times S^{3}; \Lambda) \to H_{1}(M_{0}; \Lambda) \to 0.$$
It's easy to see that $H_{2}(M_{0}; \Lambda) \cong H_{2}(N \times S^{1}; \Lambda) \cong \z_{-} \oplus \z_{-}$, $H_{1}(S^{1} \times S^{3}; \Lambda) \cong \Lambda$ and $H_{1}(M_{0}; \Lambda) \cong \z_{+}$. Therefore we have a short exact sequence
$$0 \to \z_{-} \oplus \z_{-} \to H_{2}(M; \Lambda) \to \z_{-} \to 0,$$
this implies $\pi_{2}(M) \cong H_{2}(M;\Lambda) \cong (\z_{-})^{3}$, and $g_{*} \colon \pi_{2}(M) \to \pi_{2}(Q)$ is surjective with kernel $\z_{-}^{2}$.

Since $M$ is of dimension $5$ and spin, we may represent a primitive element in the kernel by an embedded $2$-sphere with trivial normal bundle. We do (framed) surgery on this embedded $2$-sphere and obtain $g' \colon M' \to Q$. The effect of this surgery on $\pi_{2}$ can be analyzed by the following exact braid (homology with $\Lambda$-coefficients) \cite{Wall nonsimply}:

$$\xymatrix{
H_{3}(M') \ar[dr] \ar@/^2pc/[rr] & & H_{3}(W,M) \ar[dr] \ar@/^2pc/[rr]_{\partial} & & H_{2}(M) \ar[dr] \ar@/^2pc/[rr] & & 0\\
 & H_{3}(W) \ar[ur] \ar[dr] & & H_{3}(W, \partial W) \ar[ur] \ar[dr] & & H_{2}(W) \ar[ur] \ar[dr] & \\
 H_{3}(M) \ar[ur] \ar@/_2pc/[rr]^{j} & & H_{3}(W, M')  \ar[ur] \ar@/_2pc/[rr] & & H_{2}(M') \ar[ur] \ar@/_2pc/[rr] & & 0 
 }$$
where $W$ denotes the trace of the surgery. We have $H_{3}(W,M) \cong \Lambda \cong H_{3}(W,M')$, $H_{2}(W) \cong \z_{-}^{2}$,  $\ker \partial \cong \mathrm{coker} j \cong \z_{+}$. Therefore $H_{2}(M')$ is an extension of $\z_{-}^{2}$ by $\z_{+}$:
$$0 \to \z_{+} \to H_{2}(M'; \Lambda) \to \z_{-}^{2} \to 0.$$

The extension problem is solved by the following consideration: note we have a surjection $\pi_{2}(M') \to \pi_{2}(Q)$. By comparison with the Serre spectral sequence for $\cpi \to Q \to \rpi$ we see that the differential $d_{2} \colon E_{3,0}^{2} \to E_{0,2}^{2}$ in the Serre spectral sequence of $\w M' \to M' \to \rpi$ is nontrivial. If $H_{2}(M';\Lambda) \cong \z_{+} \oplus \z_{-}^{2}$, then tors$H_{2}(M;\z) \cong \z/2$, which contradicts Lemma \ref{lemma:tors}. Thus we must have $H_{2}(M';\Lambda) \cong \Lambda \oplus \z_{-}$.

The next step is to determine the kernel $K= \ker (g'_{*} \colon \pi_{2}(M') \to \pi_{2}(Q))$.  We have a short exact sequence
$0 \to K \to \Lambda \oplus \z_{-} \to \z_{-} \to 0$,
$K$ may be $\Lambda$ or $\z_{+} \oplus \z_{-}$.
Tensoring with $\z_{+}$ we get
$$\mathrm{Tor}_{1}^{\Lambda}(\z_{-}, \z_{+})=0 \to K \otimes_{\Lambda}\z_{+} \to \z \oplus \z/2 \to \z/2 \to 0.$$
From the comparison of the spectral sequences mentioned in last paragraph, we know that
$\mathrm{tors}(\pi_{2}(M') \otimes_{\Lambda} \z_{+} )\cong \z/2 \to \pi_{2}(Q) \otimes_{\Lambda}\z_{+} \cong \z/2$
must be an isomorphism.  From this it's easy to see that the assumption $K \cong \z_{+} \oplus \z_{-}$ will lead to a contradiction. Therefore $K \cong \Lambda$.

Now a standard argument in surgery shows that we may do a further surgery on an embedded $2$-sphere representing a generator of $K$ to kill the kernel $K$. The result is a normal $2$-smoothing $Y_{1}' \to \mathrm{BSpin} \times Q$.
\end{proof}

By \cite[p.~711]{Kreck}, the group $Aut(\mathrm{BSpin} \times Q)$ acts transitively on the set of normal $2$-smoothings of $Y_{1}$. On the other hand, we know that the $Aut(\mathrm{BSpin} \times Q)$-action is trivial on $\OS_{5}(Q)$, therefore  $Y_{1}'$ is not diffeomorphic to $Y_{1}$. 

Furthermore, for any $M^{5}$ homotopy equivalent to $Y_{1}$, let $h \colon M \to Y_{1}$ be a homotopy equivalence. Then the second stage Postnikov map of $M$ is given by 
$$\varphi \colon M \stackrel{h}{\longrightarrow} Y_{1} \stackrel{pr}{\longrightarrow} \rp^{2} \stackrel{i}{\longrightarrow} Q.$$
Therefore by the same argument in Lemma \ref{lem:3.3}, $[M,\varphi]=0 \in \OS_{5}(Q)$.  Thus $Y_{1}'$ is not homotopy equivalent to $Y_{1}$.

By the decomposition of the Thom spectrum $MSpin$ \cite{ABP}, we have a canonical isomorphism 
$$\OS_{5}(X) \stackrel{\cong}{\rightarrow} ko_{5}(X), \ \ [M,f] \mapsto f_{*}[M]_{ko}$$
where $[M]_{ko}$ is the $KO$-theory fundamental class of a spin manifold $M^{5}$. Therefore we have canonical isomorphisms
$$ko_{5}(Q) \cong \OS_{5}(Q) \cong \OS_{5}(S^{4}) \cong ko_{5}(S^{4}).$$
Capping with $(1,0) \in ko^{4}(S^{4}) \cong \z \oplus ko^{4}$ gives an isomorphism $ko_{5}(S^{4}) \to \z/2$. Since $ko^{4}(S^{4}) \cong ko^{4}(\mathbb H \mathrm P^{\infty})$, we obtain an element  in $ko^{4}(Q)$, denoted by $u$.  Then the isomorphism $\OS_{5}(Q) \cong \z/2$ is given by $[M,f] \mapsto \langle u, f_{*}[M]_{ko}\rangle = \langle f^{*}u, [M]_{ko}\rangle \in ko_{1} \cong \z/2$.
\end{proof}

\begin{remark}
The existence of two diffeomorphism classes in the case ($w_{1}(M)=0$, $w_{2}(M)=0$, $\pi_{2}(M)=\z_{-}$) gives the following interesting examples of embeddings: starting from an $M^{5}$ of this type, representing a generator of $\pi_{2}(M)$ by an embedded $2$-sphere and doing surgery on this sphere, we obtain a manifold $N^{5}$ with $\pi_{1}(N) \cong \z/2$, $w_{1}(N)=0$, $w_{2}(N)=0$.  A direct calculation using the  exact braid as above shows that $\pi_{2}(N) \cong \z_{+}$.  By Theorem \ref{thm:one} (1), $N \cong S^{2} \times \rp^{3}$. Reversing the surgery process, we see that $M$ can be obtained by doing surgery on an embedded $2$-sphere in $S^{2} \times \rp^{3}$ representing a generator of $\pi_{2} \cong \z$. Since there are two diffeomorphism types of such $M$, and there is no ambiguity coming from framings ($\pi_{2}(SO(3)=0$), we conclude that there are two embeddings of $S^{2}$ into $S^{2} \times \rp^{3}$ which are homotopic but not isotopic.  Note that the embedding $S^{2} \hookrightarrow S^{2} \times \rp^{3}$ is $0$-connected, this is just beyond the range of Haefliger \cite[Theorem 1 (b)]{Haefliger}, within which homotopic embeddings are isotopic.
\end{remark}

\end{document}